\documentclass[12pt,a4paper,reqno]{amsart}
\usepackage{amssymb,amsmath,amsthm}
\usepackage[T1]{fontenc}
\usepackage{breqn}
\usepackage[dvips]{graphicx}

\theoremstyle{plain}
\newtheorem{theorem}{Theorem}
\newtheorem*{polya}{P\'olya's Recurrence Theorem}
\newtheorem{corollary}{Corollary}

\newtheorem*{theorem A}{Theorem A}
\newtheorem*{theorem B}{Theorem B}
\newtheorem*{corollary A}{Corollary A}
\newtheorem*{corollary B}{Corollary B}
\theoremstyle{definition}

\newtheorem{rem}{Remark}

\newtheorem*{proof A}{Proof of Theorem A}
\newtheorem*{proof B}{Proof of Theorem B}
\newtheorem*{proof AA}{Proof of Corollary A}

\DeclareMathOperator*{\p}{\bf P}
\DeclareMathOperator*{\e}{\bf E}
\title{Quantitative recurrence results for random walks}
\author{Nuno Luzia}
\address{Universidade Federal do Rio de Janeiro, Instituto de Matem\'atica \\ Rio de Janeiro 21945-970, Brazil}  
\email{nuno@im.ufrj.br} 
\begin{document}
\maketitle
\begin{abstract}
First, we prove a \emph{local almost sure central limit theorem} for lattice random walks in the plane. 
The corresponding version for random walks in the line was considered by the author in \cite{5}. 
This gives us a quantitative version of P\'olya's Recurrence Theorem \cite{6}.
Second, we prove a \emph{local almost sure central limit theorem} for (not necessarly lattice) random walks in the line or in the plane, which will also give us quantitative recurrence results.
Finally, we prove an \emph{almost sure central limit theorem} for multidimensional (not necessarly lattice) random walks. This is achieved by exploiting a technique developed by the author in \cite{5}.
\end{abstract}

\section{Introduction and statements}
The simple random walk in $\mathbb{Z}^d$ is a path in the lattice $\mathbb{Z}^d$ such that, if we are at a given point, we move to one of its $2d$ nearest points,
each with probability $1/2d$, and so on. In other words, if $d=2$, we can move one unit up, down, to the left or to the right with equal probability. By convention, the simple random walk starts at the origin. We say that the simple random walk is \emph{recurrent} if a path returns infinitely often to the origin with probability 1. In 1921 \cite{6}, Georg P\'olia proved the following remarkable result.

\begin{polya}
The simple random walk in $\mathbb{Z}^d$ is recurrent if and only if $d=1$ or $d=2$.
\end{polya}

In this paper we address the problem of quantifying this recurrence, i.e. calculating the frequency of returns to the origin (and to other states, after a proper renormalization).

If $X=(X^1,...,X^d)$ is a random vector, we denote by $\Sigma$ the covariance matrix of $X^i$, i.e. $\Sigma=[\mathrm{cov}(X^i,X^j)]$, and by $|\Sigma|\ge 0$ its determinant (whenever it makes sense).

\subsection{Lattice random walks}
We say that the random vector $X$ is \emph{$\Lambda$-lattice valued} if there are ${\bf{h}}_1,...,{\bf{h}}_d$ linearly independent vectors in
$\mathbb{R}^d$ such that $\Lambda=\{k_1 {\bf{h}}_1+\cdots + k_d {\bf{h}}_d :  k_1,...,k_d\in\mathbb{Z}\}$ and  $\p(X\in\Lambda + {\bf{b}})=1$, for some ${\bf{b}}\in\mathbb{R}^d$  (we assume there is no sub-lattice $\Lambda'  \subset \Lambda$ with this property). For ${\bf{a}}=(a_1,...,a_d)\in\mathbb{R}^d$, let $|{\bf{a}}|=(a_1^2+\cdots a_d^2)^{-1/2}$. Let $X_i$ be i.i.d. $\Lambda$-lattice random vectors and $S_n=\sum_{i=1}^n X_i$. By abuse of notation, when we write $S_{n}={\bf{a}}\sqrt{n}$ we mean $S_n={\bf{x}}$ for some ${\bf{x}}\in\Lambda + n{\bf{b}}$ with 
$|{\bf{x}}-{\bf{a}}\sqrt{n}|\le (|{\bf{h}}_1|+\cdots |{\bf{h}}_d|)/2$ (for definiteness, we can settle for ${\bf{x}}$ one of the possible $2^d$ localizations on the lattice). We will prove the following \emph{local almost sure central limit theorem}.

\begin{theorem}\label{return1}
Let $d=2$. Let $X_i$ be i.i.d. $\Lambda$-lattice valued random vectors with ${\e}X_i={{0}}$, $|\Sigma|>0$ and ${\e}|X_i|^3<\infty$, and $S_n=\sum_{i=1}^n X_i$. Let $n_i=[i\log i]$ and ${\bf{a}}\in \mathbb{R}^2$. Then
 
\begin{equation}\label{quot}
 \frac{1}{\log \log n} \sum_{i=1}^n 1_{\{S_{n_i}={\bf{a}}\sqrt{n_i}\}}\to \frac{1}{2\pi\sqrt{|\Sigma|}} \, e^{-\frac{1}{2}{\bf{a}}^T\Sigma^{-1}{\bf{a}}} \quad\text{a.s.}
 \end{equation}
 
Let $\Delta_n^{\bf{a}}$ be the quotient between the left and right hand sides of \emph{(\ref{quot})}. Then, for every $N>0$ there exists $C>0$ such that, for every $\epsilon>0$,
\[
\sup_{{\bf{a}}\in [-N,N]^2} \p\left(|\Delta_n^{\bf{a}}-1|>\epsilon\right)\le C\epsilon^{-2}(\log \log n)^{-1}.
\]
\end{theorem}

Theorem 2 of \cite{5} is the 1-dimensional version of this theorem.

\begin{corollary}
Let $S_n$ be a random walk as in Theorem \ref{return1}. Then, for every ${\bf{a}}\in \mathbb{R}^2$,
\[
      \liminf_{n\to\infty} n^{\frac{1}{2}} \,\Big|\frac{S_n}{\sqrt{n}} - {\bf{a}}\Big|\le \max\{|{\bf{h}}_1|, |{\bf{h}}_2|\} \quad\text{a.s.}
\]
\end{corollary}

\subsection{A local almost sure central limit theorem}
Let $X_i$ be i.i.d. random vectors in $\mathbb{R}^d$ such that ${\e}X_i=0$ and $|\Sigma|>0$. By doing a linear change of variable, we may assume, without loss of generality, that ${\e}X_i=0$ and $\Sigma=I$ is the identity matrix. For ${\bf{a}}\in \mathbb{R}^d$ and $\epsilon>0$, when we write
$S_n ={\bf{a}}\pm \epsilon$ we mean $S_n\in {\bf{a}}+[-\epsilon, \epsilon]^d$.

\begin{theorem}\label{return2}
Let $d=1,2$. Let $X_i$ be i.i.d. random vectors with ${\e}X_i=0$, $\Sigma=I$ and ${\e}|X_i|^3<\infty$. Let $S_n=\sum_{i=1}^n X_i$ and ${\bf{a}}\in\mathbf{R}^d$.
\begin{enumerate}
\item[(a)] If $d=1$ then, with $1/6\le\alpha<1/2$,  $\beta=(1/2-\alpha)^{-1}$ and $n_i=[i^\beta]$,
\begin{equation}\label{densidade}
\frac{1}{\log n}\sum_{i=1}^n 1_{\{S_{n_i}={\bf{a}} \sqrt{n_i}\,\pm\, n_i^\alpha\}}\to \sqrt{\frac{2}{\pi}} \, e^{-|{\bf{a}}|^2/2} \quad\text{a.s.}
\end{equation}
Let $\Delta_n^{\bf{a}}$ be the quotient between the left and right hand sides of (\ref{densidade}). Then, for every $N>0$ there exists $C>0$ such that, for every $\epsilon>0$,
\[
\sup_{{\bf{a}}\in [-N,N]} \p\left(|\Delta_n^{\bf{a}}-1|>\epsilon\right)\le C \epsilon^{-2}(\log n)^{-4\alpha/(10\alpha+1)}.
\]
\item[(b)] If $d=2$ then, with $2/5\le\alpha<1/2$,  $\beta=(1/2-\alpha)^{-1}/2$ and $n_i=[i^\beta (\log i)^{\beta}]$,
\begin{equation}\label{densidade2}
\frac{1}{\log\log n}\sum_{i=1}^n 1_{\{S_{n_i}={\bf{a}} \sqrt{n_i}\,\pm\, n_i^\alpha\}}\to  \frac{2}{\pi} \, e^{-|{\bf{a}}|^2/2} \quad\text{a.s.}
\end{equation}
Let $\Delta_n^{\bf{a}}$ be the quotient between the left and right hand sides of (\ref{densidade2}). Then, for every $N>0$ there exists $C>0$ such that, for every $\epsilon>0$,
\[
\sup_{{\bf{a}}\in [-N,N]^2} \p\left(|\Delta_n^{\bf{a}}-1|>\epsilon\right)\le C \epsilon^{-2}(\log\log n)^{-1/5}.
\]
\end{enumerate}
\end{theorem}

\begin{corollary}
Let $S_n$ be a random walk as in Theorem \ref{return2}.
\begin{enumerate}
\item[(a)] If $d=1$ then,  for every ${\bf{a}}\in \mathbb{R}$,
\[
      \liminf_{n\to\infty} n^{\frac{1}{3}} \,\Big|\frac{S_n}{\sqrt{n}} - {\bf{a}}\Big|=0 \quad\text{a.s.}
\]
\item[(b)] If $d=2$ then, for every ${\bf{a}}\in \mathbb{R}^2$,
\[
      \liminf_{n\to\infty} n^{\frac{1}{10}} \,\Big|\frac{S_n}{\sqrt{n}} - {\bf{a}}\Big|=0 \quad\text{a.s.}
\]
\end{enumerate}
\end{corollary}

For the dynamical systems results counterpart of this and previous sections, we refer the reader to \cite{3} and \cite{5}.

\subsection{A multidimensional almost sure central limit theorem}
In this section $d\in\mathbb{N}$.

\begin{theorem}\label{return3}
Let $X_i$ be i.i.d. random vectors with ${\e}X_i=0$, $|\Sigma|>0$ and ${\e}|X_i|^3<\infty$, and $S_n=\sum_{i=1}^n X_i$. Let ${\bf{a}}\in\mathbf{R}^d$, $\epsilon>0$ and $n_i$ be an increasing sequence of positive integers satisfying 
\begin{equation*}\label{seqexp}
      n_{i+1}/ n_i \ge1+A i^{-\alpha}\quad \text{for all }i 
\end{equation*}
where $A>0$, $0<\alpha<1$ if $d=1,2$, and $0<\alpha<(d/2-1)^{-1}$ if $d>2$. Then
\begin{equation}\label{densidade3}
\frac{1}{n}\sum_{i=1}^n 1_{\{S_{n_i}=({\bf{a}} \,\pm \,\epsilon) \sqrt{n_i}\}} \to \frac{1}{\sqrt{(2\pi)^{d} |\Sigma|}}\, \int_{{\bf{a}}\,\pm\,\epsilon} e^{-\frac{1}{2}{\bf{x}}^T\Sigma^{-1}{\bf{x}}}\,d{\bf{x}} \quad\text{a.s.}
\end{equation}

Let $\Delta_n^{\bf{a}}$ be the quotient between the left and right hand sides of (\ref{densidade3}). Then, for every $N>0$ there exists $C>0$ such that, for every $\epsilon>0$,
\[
\sup_{{\bf{a}}\in [-N,N]^{d}} \p\left(|\Delta_n^{\bf{a}}-1|>\epsilon\right)\le C \epsilon^{-2} n^{\alpha(d/2-1)-1} \log n.
\]
\end{theorem}

\begin{rem}
With respect to the case $d=1$, in \cite{5} we have stated this result improperly. So here we correct the statement.
\end{rem}

In particular, considering the sequence $n_i=2^i$ and the `change of variable' $k=2^i$ we get the following \emph{multidimensional almost sure central limit theorem}.
\begin{corollary}
With the same hypotheses as Theorem \ref{return3},
\begin{equation*}
\frac{1}{\log n}\sum_{k=1}^n 1_{\{S_{k}=({\bf{a}}\, \pm\, \epsilon) \sqrt{k}\}} \frac{1}{k} \to \frac{1}{\sqrt{(2\pi)^{d} |\Sigma|}}\, \int_{{\bf{a}}\,\pm\,\epsilon} e^{-\frac{1}{2}{\bf{x}}^T\Sigma^{-1}{\bf{x}}}\,d{\bf{x}} \quad\text{a.s.}
\end{equation*}
\end{corollary}

\subsection{Commentaries}
Even though we have restricted our attention to i.i.d. random vectors $X_i$, so that the random walk $S_n=\sum_{i=1}^n X_i$ is in the domain of attraction of a multivariate
gaussian law, it is clear from the proofs that we can apply our method to more general random walks. If we only consider independent (or even \emph{weakly dependent}) random vectors $X_i$ such that the corresponding random walk $S_n$ is in the domain of attraction of some stable law, and we have an appropriate (local) limit theorem with rates of convergence for this law (see \cite{7} and references therein for these kind of results), then our method can apply.    

\section{Proofs} 

\begin{proof}[Proof of Theorem \ref{return1}] The idea of the proof is the same as the one used in the proof of Theorem 2 of \cite{5} but with additional difficulties.\text{ }\\
\emph{The simple random walk.}
We treat this particular case separately because its proof is elementary. 
Then we say how it extends easily to lattice random walks by using a local central limit theorem with rates.

In this particular case, $S_n={\bf{a}} \sqrt{n}$ means $S_n=(2[a_1\sqrt{n}/2], 2[a_2\sqrt{n}/2])$. Note that if $T_n^1$ and $T_n^2$ are independent simple random walks in the line, then the simple random walk in the plane $S_n$ is obtainned by rotating
$(T_n^1, T_n^2)$ 45 degrees and then dividing by $\sqrt{2}$. So, for even $n$,
\[
    \p(S_n={\bf{a}} \sqrt{n})= \p(T_n^1=(a_1+a_2) \sqrt{n}\pm2) \p(T_n^2=(-a_1+a_2) \sqrt{n}\pm2).
\]
By using Stirling's formula
\[  
        n!=\sqrt{2\pi n} \left(\frac{n}{e}\right)^n \left(1+O\left(\frac{1}{n}\right)\right)
\]
where $O(n^{-1})>0$, together with the simple inequalities
\[ 
      (1+k/m)^m\le e^k \le (1+k/m)^m (1+k^2/m),
\]
\[
     (1+k/m)^m \le (1-k/m)^{-m}\le (1+k/m)^m  (1-k^2/m)^{-1},
\]
where $k\ge0,\,m> k^2$ and, for the second line of inequalities, we also assume $m\ge 1,\,m> k$, we get the following local central limit theorem with rates:
\begin{align}
&\p(T_{n}=a\sqrt{n})=\binom{n}{\left(n+2[a\sqrt{n}/2]\right)/2} 2^{-n}\notag  \\
&= \sqrt{\frac{2}{\pi}} \,n^{-1/2} \, e^{-a^2/2} \left(1+O\left(\frac{|a|^3}{\sqrt{n}}+\frac{1}{n}\right)\right), \label{localclt1}
\end{align}
for even $n>64 a_+^6$, where $a_+=\max\{1,a\}$ and the constants involved in $O(\cdot)$ do not depend on $a$. Note that (\ref{localclt1}) remains unchanged if we use 
$a\sqrt{n}+O(1)$ instead of $a\sqrt{n}$.
Then
\begin{align}
\p(S_{n}={\bf{a}}\sqrt{n})= \frac{2}{\pi} \,n^{-1} \, e^{-|{\bf{a}}|^2} \left(1+O\left(\frac{|{\bf{a}}|^3}{\sqrt{n}}+\frac{1}{n}\right)\right), \label{localclt2}
\end{align}
for even $n>512 |{\bf{a}}|_+^6$, and also (\ref{localclt2}) remains unchanged if we use ${\bf{a}}\sqrt{n}+\vec{O}(1)$ instead of ${\bf{a}}\sqrt{n}$, 
where $\vec{O}(1)=(O(1), O(1))$.

Let  $\tilde{S}_n=\sum_{i=1}^n 1_{\{S_{n_i}={\bf{a}}\sqrt{n_i}\}}$
where $n_i=2[i\log i]$. By Theorem 1 of \cite{5}, since  ${\e}\tilde{S}_n \to \infty$, if we prove that
\begin{equation}\label{rest}
      \mathrm{var}(\tilde{S}_n)= O\left(\frac{({\e}\tilde{S}_n)^2}{(\log {\e}\tilde{S}_n) (\log\log {\e}\tilde{S}_n)^\gamma}\right)
\end{equation}
for some $\gamma>1$, then      
\[      
      \frac{\tilde{S}_n}{{\e}\tilde{S}_n}\to 1  \quad \text{a.s.}
\]
which gives our result. Note that the random variables $1_{\{S_{i}={\bf{a}}\sqrt{i}\}}$ are not independent but we are going to prove that 
$Z_{i}=1_{E_i}$, where $E_i={\{S_{n_i}={\bf{a}}\sqrt{n_i}\}}$, are \emph{quasi-independent} in the sense of (\ref{rest}).
      
To prove condition (\ref{rest}) we have $\sum_{i=1}^n \mathrm{var}( Z_i)\le{\e}\tilde{S}_n$ and, for $i<j$,
\begin{align*}
     \p(E_i\cap E_j)&=\p(S_{n_i}={\bf{a}}\sqrt{n_i})\p(S_{n_j-n_i}={\bf{a}}(\sqrt{n_j}-\sqrt{n_i})+\vec{O}(1))\\
 &=\p(E_i)\p(E_j) \, \frac{n_j}{n_j-n_i}\, R,
\end{align*}
where
\[
    R= e^{2|{\bf{a}}|^2 \frac{\sqrt{n_i}}{\sqrt{n_j}+\sqrt{n_i}}} \left(1+O\left(\frac{1}{n_j-n_i}+\frac{{|\bf{a}}|^3}{\sqrt{n_j-n_i}}\right)\right)
\]
is obtainned using (\ref{localclt2}), valid for $n_j-n_i>C_1|{\bf{a}}|_+^6$, and $C_1, C_2,...$ denote absolut constants 
(which do not depend on ${\bf{a}}$). Note that $n_j-n_i\ge n_{i+1}-n_i>\log i>C_1|{\bf{a}}|_+^6$ for $i>e^{C_1|{\bf{a}}|_+^6}$.
To estimate
\begin{equation}\label{erdos11}
\Bigl|\sum_{1\le i<j\le n}  \p(E_i\cap E_j)-\p(E_i)\p(E_j)\Bigr|
\end{equation}
we separate the sum into two cases. Let 
\[
       \sqrt{\nu}_n=\log \log n > (\log {\e}\tilde{S}_n)^2. 
\]
In all cases we restrict to $i>e^{C_1|{\bf{a}}|_+^6}$ and $\nu_n\ge 2C_1|{\bf{a}}|_+^6$.

\emph{Case 1}: $n_j>\nu_n n_i$\\
Then we see that 
\[  
    \frac{n_j}{n_j-n_i}\le 1+\frac{2}{\nu_n} \;\text{ and }\; R= 1+O\left(\frac{|{\bf{a}}|^2}{\sqrt{\nu_n}}+\frac{1}{n_j-n_i}+\frac{|{\textbf{a}}|^3}{\sqrt{n_j-n_i}}\right).
\]
Since $\sum_{i=1}^\infty \p(E_i)/\sqrt{n_i}\le C_2<\infty$, this implies this case contribution of (\ref{erdos11}) is less than some constant $C_3$ times
\[
 \frac{({\e}\tilde{S}_n)^2}{\nu_n} + |{\bf{a}}|^2\frac{({\e}\tilde{S}_n)^2}{\sqrt{\nu_n}} + \frac{{\e}\tilde{S}_n}{\nu_n} + 
 |{\bf{a}}|^3 \frac{{\e}\tilde{S}_n}{\sqrt{\nu_n}}.
\]

\emph{Case 2}: $n_j\le\nu_n n_i$\\
Clearly this case contribution of (\ref{erdos11}) is less than
\begin{align}
    &C_4 e^{|{\bf{a}}|^2} \sum_{i,j}\p(E_i)\p(E_j) \frac{n_j}{n_j-n_i} \notag \\
    &\le C_5 \sum_{i=1}^n \p(E_i) \sum_j (n_j-n_i)^{-1}\label{erdos2}
\end{align}
Given $i$, let $N$ be the number of $j$'s satisfying $n_i\le n_j \le \nu_n n_i$. Then $n_{i+N}\le \nu_n n_i$ and $N+i\le \nu_n i $. Also $n_j-n_i>\log i (j-i)$. Then 
\begin{align*}
    \sum_{j=i+1}^{N+i}(n_j-n_i)^{-1}\le \frac{\log \nu_n}{\log i}+1\end{align*}
and  (\ref{erdos2}) is less than some constant $C_6$ times
\[
        {\e}\tilde{S}_n + \log \nu_n=O\left(({\e}\tilde{S}_n)^2/\sqrt{\nu_n}\right), 
\]
which ends the proof.

\emph{Lattice random walks.} 
Since ${\e}|X_i|^3<\infty$, we have the following local central limit theorem with rates (see \cite{2}):
\begin{equation}\label{lclt1}
   \p(S_n={\bf{a}}\sqrt{n})= n^{-1}\frac{1}{2\pi\sqrt{|\Sigma|}}\, e^{-\frac{1}{2}{\bf{a}}^T\Sigma^{-1}{\bf{a}}} \left(1+ O\left(\frac{1}{\sqrt{n}}\right)\right),
\end{equation}
where the constants involved in $O(\cdot)$) might depend (continuously) on ${\bf{a}}, |\Sigma|$ and on the distribution of $X_i$.
Then the proof is similar to the simple random walk.

The uniform convergence in probability follows from Chebyshev's inequality
\begin{equation}\label{cheb}
 \p(|S_n-{\e}S_n|>\epsilon{\e}S_n)\le\mathrm{var}(S_n)/(\epsilon{\e}S_n)^2\le C \epsilon^{-2} (\nu_n)^{-1/2},
\end{equation}
and the fact that $C$ can be chosen uniformly in ${\bf{a}}\in [-N,N]^2$.
\end{proof}

\begin{proof}[Proof of Theorem \ref{return2}]
By the Berry-Esseen Theorem and its multidimensional version (see \cite {4} and \cite{1}),
\begin{equation}\label{beress}
 \left| \p\left(\frac{S_n}{\sqrt{n}}={\bf{a}}\pm\epsilon \right) -  \frac{1}{\sqrt{(2\pi)^{d}}}\,\int_{{\bf{a}}\,\pm\,\epsilon} e^{-|{\bf{x}}|^2/2}\,d{\bf{x}}  \right|\le C\rho/ \sqrt{n},
\end{equation}
where $\rho={\e}|X_i|^3$ and $C$ is an absolute constant. We will interpret it as
\begin{equation}\label{lclt3}
   \p(S_n={\bf{a}}\sqrt{n}\pm n^\alpha)=(2/\pi)^{d/2}\,n^{d(\alpha-1/2)} \varphi_n({\bf{a}}) \left(1+ O(n^{-1/2 - d(\alpha-1/2)})\right),
\end{equation}
where
\[
      \varphi_n({\bf{a}})=2^{-d}\, n^{-d(\alpha-1/2)}\, \int_{{\bf{a}}\,\pm \,n^{\alpha-1/2}}  e^{-|{\bf{x}}|^2/2}\,d{\bf{x}}.
\]   
Hereafter, $C_1, C_2,...$ will denote absolute constants that might depend (continuously) on ${\bf{a}}$ and $\rho$ (this includes the constants involved in
 $O(\cdot)$).
 
Let $E_i$ denote the event $S_{n_i}={\bf{a}}\sqrt{n_i}\,\pm\, n_i^{\alpha}$ and consider the random variables $Z_i=1_{E_i}$ and
$\tilde{S}_n=\sum_{i=1}^n Z_i$. As before, we want to apply Theorem 1 of \cite{4}, so we must prove (\ref{rest}).
Clearly, (\ref{lclt3}) implies ${\e}\tilde{S}_n\sim  (2/ \pi)^{d/2} \,e^{-|{\bf{a}}|^2/2}\,(\log_d n)$, where $\log_1n=\log n$ and $\log_2 n = \log \log n$. 
Also $\sum_{i=1}^n \mathrm{var}( \tilde{X}_i)\le{\e}\tilde{S}_n$ and, for $i<j$,
$\p(E_i\cap E_j)$ is less than or equal to
\begin{align*}
     &\p(S_{n_i}={\bf{a}}\sqrt{n_i}\,\pm\, n_i^\alpha)\p(S_{n_j-n_i}={\bf{a}}(\sqrt{n_j}-\sqrt{n_i})\,\pm\, (n_j^\alpha+n_i^\alpha))\\
     &=\p(E_i)\p(E_j) R_1,
\end{align*}
and greater than or equal to
\begin{align*}
     &\p(S_{n_i}={\bf{a}}\sqrt{n_i}\,\pm\, n_i^\alpha)\p(S_{n_j-n_i}={\bf{a}}(\sqrt{n_j}-\sqrt{n_i})\,\pm\, (n_j^\alpha-n_i^\alpha))\\
     &=\p(E_i)\p(E_j) R_2.
\end{align*}
Using the real Taylor expansion 
\[
\int_{\tilde{a}-x}^{\tilde{a}+x} e^{-s^2/2}\, ds = 2x e^{-\tilde{a}^2/2} + O(x^2),
\]
valid for all $\tilde{a}, x\in\mathbb{R}$, with this $O(\cdot)$ independent of $\tilde{a}$, we also get, by Fubini theorem,
\[
\int_{\tilde{\bf{a}}\pm x}e^{-|{\bf{s}}|^2/2}\, d{\bf{s}} = 2^d x^d e^{-|\tilde{\bf{a}}|^2/2}\left(1 + \sum_{k=1}^d O(x^k)\right),
\]
for all $x\in\mathbb{R}$ and $\tilde{\bf{a}}$ in a bounded domain. Using this together with (\ref{beress}) we get
\begin{multline*}
\p \left( \frac{S_{n_j-n_i}} {\sqrt{n_j-n_i}} =\tilde{\bf{a}}\pm x\right)=(2/\pi)^{d/2} x^d  e^{-|\tilde{\bf{a}}|^2/2} \\
  \left(1 + \sum_{k=1}^d O(x^k)+O(x^{-d} (n_j-n_i)^{-1/2})\right), 
\end{multline*}
for all $x\in\mathbb{R} - 0$ and $\tilde{\bf{a}}$ in a bounded domain. This and (\ref{lclt3}) gives us the estimates 
\begin{multline*}
    R_\mu = e^{\frac{|{\bf{a}}|^2\sqrt{n_i}}{\sqrt{n_j}+\sqrt{n_i}}} \left(\frac{n_j}{n_j-n_i}\right)^{d/2} \Bigl(1+O\Bigl(n_j^{\alpha-1/2} + n_j^{-1/2-d(\alpha-1/2)}+ 
    \Bigl(\frac{n_i}{n_j}\Bigr)^\alpha \\ +\sum_{k=1}^d \left(\frac{n_j^\alpha}{\sqrt{n_j-n_i}}\right)^k+\frac{(n_j-n_i)^{(d-1)/2}}{(n_j^\alpha-n_i^\alpha)^{d}} \Bigr) \Bigr).
\end{multline*}
To estimate
\begin{equation}\label{erdos11b}
\Bigl|\sum_{1\le i<j\le n}  \left(\p(E_i\cap E_j)-\p(E_i)\p(E_j)\right)\Bigr|
\end{equation}
we separate the sum in two cases. Let 
\begin{equation*}
\nu_n=\begin{cases} (\log n)^{(5\alpha/2+1/4)^{-1}} &\text{ if } d=1\\ (\log \log n)^{1/5\alpha} &\text{ if } d=2. \end{cases}
\end{equation*}
In any case  $\nu_n^\alpha> (\log {\e}\tilde{S}_n)^2$.\\
\emph{Case 1}: $n_j>\nu_n n_i$\\
Then $R_k=1+O(\nu_n^{-\alpha}+n_j^{\alpha-1/2}+ n_j^{-1/2-d(\alpha-1/2)})$. It is easy to see that both 
\[
     \sum_{1\le i<j\le n} \p(E_i)\p(E_j) n_j^{\alpha-1/2}, \quad \sum_{1\le i<j\le n} \p(E_i)\p(E_j) n_j^{-1/2-d(\alpha-1/2)}
\]
are  $O({\e}\tilde{S}_n)$ because $(1-d^{-1})/2<\alpha<1/2$. So this case contribution of (\ref{erdos11b}) is $O(({\e}\tilde{S}_n)^2/\nu_n^\alpha)$.

\emph{Case 2}: $n_j\le\nu_n n_i$\\
This case contribution of (\ref{erdos11b}) is less than $C_1$ times $\sum_{k=1}^d I_k+II$ where
\begin{align*}
    I_k &= \sum_{i,j}\p(E_i)\p(E_j) \left(\frac{n_j}{n_j-n_i}\right)^{d/2} \left( 1+ O\left( \frac{n_j^\alpha}{\sqrt{n_j-n_i}}\right)^k\right),\\
    II &= \sum_{i,j}\p(E_i)\p(E_j) \left(\frac{n_j}{n_j-n_i}\right)^{d/2} \left( 1+ O\left(\frac{(n_j-n_i)^{(d-1)/2}}{(n_j^\alpha-n_i^\alpha)^{d}} \right)\right).
\end{align*}
Given $i$, let $N$ be the number of $j$'s satisfying $n_i\le n_j \le \nu_n n_i$. Then $n_{i+N}\le \nu_n n_i$ and $N+i\le \nu_n^{\beta^{-1}} i $.
We have that
\begin{equation}\label{crit}
\frac{n_j^\alpha}{\sqrt{n_j-n_i}}\le C_2 \nu_n^{\alpha}\, i^{\beta(\alpha-1/2)+1/2}\,(j-i)^{-1/2},
\end{equation}
where $\beta(\alpha-1/2)+1/2\le0$.\\
\emph{Case 2.1}: $d=1$\\
Then
\begin{align*}
    I_1&\le C_3 \nu_n^\alpha \sum_{i,j}\p(E_i)\p(E_j) \left(\frac{n_j}{n_j-n_i}\right)^{1/2}
    \le C_4\nu_n^\alpha \sum_{i,j}\p(E_i) \frac{n_j^\alpha}{\sqrt{n_j-n_i}}\\
    &\le C_4 \nu_n^{2\alpha}\sum_{i,j}\p(E_i) i^{\beta(\alpha-1/2)+1/2}\, (j-i)^{-1/2}\\
     &\le C_5 \nu_n^{2\alpha+\beta^{-1}/2}\sum_{i}\p(E_i) i^{\beta(\alpha-1/2)+1}.
\end{align*}
Since $\beta(\alpha-1/2)+1=0$, we get $I\le C_5 \nu_n^{3\alpha/2+1/4}\, {\e}\tilde{S}_n=O(({\e}\tilde{S}_n)^2/\nu_n^\alpha)$. If 
$1-\alpha\beta\le0\Leftrightarrow \alpha\ge1/4$ then
\begin{equation*}
         (n_j^\alpha-n_i^\alpha)^{-1}\le C_6 i^{1-\alpha\beta} (j-i)^{-1},
\end{equation*}
and
\begin{align*}
    II\le C_7  \sum_{i,j}\p(E_i)\p(E_j) \left(\frac{n_j}{n_j-n_i}\right)^{1/2},
\end{align*}
which we already have seen to be $O(({\e}\tilde{S}_n)^2/\nu_n^\alpha)$. If $1/6\le\alpha<1/4$ then
\begin{equation*}\label{crit2}
         (n_j^\alpha-n_i^\alpha)^{-1}\le C_8 \nu_n^{\beta^{-1}-\alpha}\, i^{1-\alpha\beta} (j-i)^{-1},
\end{equation*}
and
\begin{align*}
    II&\le   \sum_{i,j}\p(E_i)\p(E_j) \left(\frac{n_j}{n_j-n_i}\right)^{1/2}  \left(1 + C_9 \nu_n^{\beta^{-1}-\alpha}  i^{1-\alpha\beta} (j-i)^{-1}\right)\\
     &\le O(({\e}\tilde{S}_n)^2/{\nu_n}^\alpha) +  C_{10}\nu_n^{\beta^{-1}-\alpha} \sum_{i,j}\p(E_i) \frac{n_j^\alpha}{\sqrt{n_j-n_i}} \,i^{1-\alpha\beta}\,(j-i)^{-1}\\
    &\le  O(({\e}\tilde{S}_n)^2/{\nu_n}^\alpha) + C_{11} \nu_n^{\beta^{-1}}\sum_{i,j}\p(E_i) i^{-\beta/2+3/2}\, (j-i)^{-3/2},
\end{align*}
where we used (\ref{crit}). Now $\alpha\ge 1/6\Leftrightarrow\beta\ge 3 $, so $II=O(({\e}\tilde{S}_n)^2/{\nu_n}^\alpha)$.\\
\emph{Case 2.2}: $d=2$\\
Now
\begin{equation*}\label{crit2}
\frac{n_j^\alpha}{\sqrt{n_j-n_i}}\le C_{12} \nu_n^{\alpha}\, (\log i)^{\beta(\alpha-1/2)}\,(j-i)^{-1/2}.
\end{equation*}
Then
\begin{align*}
    I_1&\le C_{12} \nu_n^\alpha \sum_{i,j}\p(E_i)\p(E_j)\frac{n_j}{n_j-n_i}
    \le C_{13}\nu_n^\alpha \sum_{i,j}\p(E_i) \frac{n_j^{2\alpha}}{n_j-n_i}\\
    &\le C_{14} \nu_n^{3\alpha}\sum_{i,j}\p(E_i) (\log i)^{-1}\, (j-i)^{-1}\le C_{14} \nu_n^{3\alpha}\left( \log \nu_n + {\e}\tilde{S}_n\right),
\end{align*}
which is $O(({\e}\tilde{S}_n)^2/\nu_n^\alpha)$. The same computation shows that $I_2=O(({\e}\tilde{S}_n)^2/\nu_n^\alpha)$.
Now
\[
       \frac{\sqrt{n_j-n_i}}{(n_j^\alpha-n_i^\alpha)^2}\le C_{15} \nu_n^{2\alpha}\, i^{\beta/2-2\alpha\beta+3/2}\, (\log i)^{\beta(1/2-2\alpha)}\, (j-i)^{-3/2}.
\]
Note that $\beta/2-2\alpha\beta+3/2\le0\Leftrightarrow \alpha\ge 2/5$ and $1/2-2\alpha<0$. Then, proceeding as for $I_1$, we get 
$II=O(({\e}\tilde{S}_n)^2/\nu_n^\alpha)$.
 
So Case 2 contribution of (\ref{erdos11b}) is also $O(({\e}\tilde{S}_n)^2/\nu_n^\alpha)$.

The uniform convergence in probability follows from the same arguments used in (\ref{cheb}).

\end{proof}

\begin{proof}[Proof of Theorem \ref{return3}]
By the Berry-Esseen Theorem and its multidimensional version (see \cite {4} and \cite{1}) we have
\begin{equation}\label{beress2}
 \p\left(\frac{S_n}{\sqrt{n}}={\bf{a}}\pm \epsilon \right)=  (2\pi)^{-d/2} \int_{{\bf{a}}\pm\epsilon}  e^{-\frac{1}{2}{\bf{x}}^T\Sigma^{-1}{\bf{x}}}\,d{\bf{x}} 
 \left(1+O\left(\frac{1}{\sqrt{n}}\right)\right).
\end{equation}
Hereafter, $C_1, C_2, ...$ will denote absolute constants that might depend (continuously) on ${\bf{a}}$, $\rho={\e}|X_i|^3$, $d$ and the eigenvalues of $\Sigma$ (this includes the constants involved in $O(\cdot)$).
 
Let $E_i$ denote the event $S_{n_i}=({\bf{a}}\pm\epsilon)\sqrt{n_i}$ and consider the random variables $Z_i=1_{E_i}$ and
$\tilde{S}_n=\sum_{i=1}^n Z_i$. As before, we want to apply Theorem 1 of \cite{3}, so we must prove (\ref{rest}).
Clearly (\ref{beress}) implies ${\e}\tilde{S}_n\,\sim n (2\pi)^{-d/2} \int_{{\bf{a}}\pm\epsilon}  e^{-\frac{1}{2}{\bf{x}}^T\Sigma^{-1}{\bf{x}}}\,d{\bf{x}} \to\infty$. Also 
$\sum_{i=1}^n \mathrm{var}( \tilde{X}_i)\le{\e}\tilde{S}_n$ and, for $i<j$,
$\p(E_i\cap E_j)$ is less than or equal to
\begin{align*}
     &\p(S_{n_i}=({\bf{a}}\pm\epsilon)\sqrt{n_i})\p(S_{n_j-n_i}={\bf{a}}(\sqrt{n_j}-\sqrt{n_i})\pm \epsilon(\sqrt{n_j}+\sqrt{n_i}))\\
     &=\p(E_i)\p(E_j) R_1,
\end{align*}
and greater than or equal to
\begin{align*}
     &\p(S_{n_i}=({\bf{a}}\pm\epsilon)\sqrt{n_i})\p(S_{n_j-n_i}=({\bf{a}}\pm\epsilon)(\sqrt{n_j}-\sqrt{n_i}))\\
     &=\p(E_i)\p(E_j) R_2.
\end{align*}
To estimate $R_k$ we use (\ref{beress2}) and then, by separating an integral over a cube into several integrals over appropriate boxes, and then using the mean value theorem, we get
\begin{align*}
    R_k= 1+ O \left( \left( 1- \frac{\sqrt{n_j}\pm\sqrt{n_i}}{\sqrt {n_j-n_i}} \right)  \left(\frac{\sqrt{n_j}}{\sqrt {n_j-n_i}} \right)^{d-1} + \frac{1}{\sqrt{n_j-n_i}}\right). 
\end{align*}
To estimate
\begin{equation}\label{erdos11c}
\Bigl|\sum_{1\le i<j\le n}  \p(E_i\cap E_j)-\p(E_i)\p(E_j)\Bigr|
\end{equation}
we separate the sum in two cases. Let 
\[
      \nu_n=\frac{n^{1-\alpha(d/2-1)}}{\log n} > (\log {\e}\tilde{S}_n)^2.
\]
\emph{Case 1}: $n_j>\nu_n n_i$\\
Then $R_k=1+O(\nu_n^{-1})$ and this case contribution of (\ref{erdos11c}) is $O(({\e}\tilde{S}_n)^2/\nu_n^\alpha)$.

\emph{Case 2}: $n_j\le\nu_n n_i$\\
Clearly this case contribution of (\ref{erdos11c}) is less than
\begin{align}\label{as2}
    C_1\sum_{i,j}\p(E_i)\p(E_j) \left(\frac{\sqrt{n_j}}{\sqrt {n_j-n_i}} \right)^{d}.
\end{align}
The hypothesis on the sequence $n_i$ implies
\[
    \left(\frac{\sqrt{n_j}}{\sqrt {n_j-n_i}} \right)^{d}\le C_2 i^{d\alpha/2}.
\]
Given $i$, let $N$ be the number of $j$'s satisfying $n_i\le n_j \le \nu_n n_i$. Then $n_{i+N}\le \nu_n n_i$ and using the hypothesis on the sequece $n_i$ we get
\begin{equation*}
 \sum_{k=i}^{N+i+1} i^{-\alpha}\le C_3 \log \nu_n,
\end{equation*}
and simple calculus shows that $N\le C_4 i^{-\alpha}\log \nu_n$ (here we used $0<\alpha<1$). Then (\ref{as2}) becomes less than $C_5 \log \nu_n \,n^{1+\alpha(d/2-1)}$
which is $O(({\e}\tilde{S}_n)^2/\nu_n)$.

The uniform convergence in probability follows from the same arguments used in (\ref{cheb}).
\end{proof}

\end{document}